\documentclass{amsart}

\usepackage{graphicx}
\usepackage{epsfig}
\usepackage{epstopdf}
\usepackage{psfrag}

\usepackage{amssymb}
\usepackage{amsmath}
\usepackage{amsfonts}
\usepackage{amsthm}

\usepackage{enumerate}
\usepackage{color}
\usepackage{setspace}

\usepackage[mathscr]{eucal}

\input xy
\xyoption{all}
\xyoption{arc}
\xyoption{tile}
\xyoption{import}
\xyoption{2cell}
\xyoption{poly}
\xyoption{web}
\xyoption{knot}

\newcommand{\gc}{\chi}

\renewcommand{\ge}{\varepsilon}

\newcommand{\gl}{\lambda}
\newcommand{\gm}{\mu}

\newcommand{\gs}{\sigma}
\newcommand{\gt}{\tau}

\newcommand{\gG}{\Gamma}
\newcommand{\gL}{\Lambda}

\newcommand{\gS}{\Sigma}

\newcommand{\gX}{\Xi}

\newcommand{\bb}{\mathbb}
\newcommand{\mc}{\mathcal}

\newcommand{\overtie}[2]{             
\begin{aligned} \displaystyle         %
\operatornamewithlimits{              %
 \begin{aligned} #1                   %
 \end{aligned} }_{#2}                 %
\end{aligned} }

\usepackage{amssymb}

\usepackage{graphicx}

\usepackage{color}

\newtheorem{theorem}{Theorem}[section]
\newtheorem{lemma}[theorem]{Lemma}

\theoremstyle{definition}
\newtheorem{definition}[theorem]{Definition}
\newtheorem{example}[theorem]{Example}

\newtheorem{proposition}[theorem]{Proposition}
\newtheorem{corollary}[theorem]{Corollary}

\theoremstyle{remark}
\newtheorem{remark}[theorem]{Remark}

\theoremstyle{problem}

\def\CC{\mathbb{C}}

\def\QQ{\mathbb{Q}}

\def\ZZ{\mathbb{Z}}

\newcommand{\bs}{\backslash}

\def\cal{\mathcal}  
\newcommand{\A}{{\cal A}}
\newcommand{\B}{{\cal B}}
\newcommand{\C}{{\cal C}}

\newcommand{\Ho}{{\cal H}}

\numberwithin{equation}{section}

\newcommand{\codim}{\operatorname{codim}}

\newcommand{\Hom}{\operatorname{Hom}}

\begin{document}

\title[Edge colored hypergraphic Arrangements]{Edge colored hypergraphic arrangements}


\author{Matthew Miller}
\address{Department of Mathematics, Bucknell University, Lewisburg, PA     ,USA}
\curraddr{}
\email{matthew.miller@bucknell.edu}

\author{Max Wakefield}
\address{Department of Mathematics, Hokkaido University, Sapporo 060-0810, Japan}
\curraddr{}
\email{wakefield@math.sci.hokudai.ac.jp}

\thanks{The second author has been supported by NSF grant \# 0600893
and the NSF Japan program.
}

\subjclass[]{}

\date{}

\dedicatory{}

\begin{abstract} A subspace arrangement defined by intersections of hyperplanes of the braid arrangement can be encoded by an edge colored hypergraph. It turns out that the characteristic polynomial of this type of subspace arrangement is given by a generalized chromatic polynomial of the associated edge colored hypergraph. The main result of this paper supplies a sufficient condition for the existence of non-trivial Massey products of the subspace arrangements complex complement. This is accomplished by studying a spectral sequence associated to the Lie coalgebras of Sinha and Walter. \end{abstract}

\maketitle



\section{Introduction}

Let $V$ be a complex vector space of dimension $\ell$. In this paper, a subspace arrangement $\A$ is a finite collection of affine subspaces of $V$ and we assume that there are no  inclusions between elements of $\A$. Let $L(\A)$ be the labeled intersection lattice of $\A$ defined by all possible non-empty intersections of elements from $\A$ ordered by reverse inclusion where the label is defined by codimension in $V$.  We call $M(\A ):=V\bs \bigcup\limits_{X\in \A} X$ the complement of $\A$. Choose $\{x_1,\ldots,x_\ell\}$ to be a basis for $V^*$. Let $\A_\ell$ be the arrangement of hyperplanes in $V$ defined by the linear forms $x_i-x_j$ for all $1\leq i<j\leq \ell$, which is called the braid arrangement or the Coxeter arrangement of type $\text{A}_\ell$. The intersection lattice $L(\A_\ell )$ of $\A_\ell$ is naturally isomorphic to the partition lattice. Following the terminology of Bj\"orner in \cite{Bj-Subspaces} we say that a subspace arrangement $\A$ is embedded in a hyperplane arrangement $\B$ if $\A\subseteq L(\B )$. The purpose of this paper is to begin studying combinatorial and topological properties of arbitrary subspace arrangements embedded in $\A_\ell$.

In the last thirty years a considerable amount of attention was given to certain types of subspace arrangements embedded in $\A_\ell$. These include graphic hyperplane arrangements, hypergraph arrangements or diagonal arrangements, orbit arrangements, and k-equal arrangements.  Their study involves many areas of mathematics, such as combinatorics, algebra, topology, and even computational complexity theory (see \cite{Bj-Subspaces}, \cite{BLY-k-equal}, \cite{BW95}, \cite{BS-Charac}, \cite{Koz97}, \cite{Koz-hyper}, \cite{LiLi-geners}, \cite{PRW-diag}, and \cite{Yuz-Small}). Even the most general of these, the class of hypergraphic arrangements, is a much smaller class than that of all possible subspace arrangements embedded in $\A_\ell$. We study this larger class of arrangements by associating an edge colored hypergraph to each such subspace arrangement.

The celebrated results of Goresky and MacPherson in \cite{GM88} concerning combinatorial formulas for the cohomology of the real complement of an arbitrary subspace arrangement led to the development and computation of topological data for many special families of subspace arrangements. In particular, Bj\"orner and Welker in \cite{BW95} give formulas for the Betti numbers and show some non-vanishing results for higher homotopy groups of the real and complex complements of $k$-equal arrangements. In \cite{Yuz-Small} Yuzvinsky lists generators and relations for the cohomology algebra for the complex complement of $k$-equal arrangements.

Recently, the question of formality of the complex complement of a subspace arrangement has received more attention. In \cite{FY-Formal} Feichtner and Yuzvinsky prove that the wonderful models of De Concini and Procesi from \cite{DCP95} are quasi-isomorphic to Yuzvinsky's relative atomic complexes from \cite{Yuz-Small}. Then they show that this relative atomic complex is a formal differential graded algebra when the subspace arrangements' intersection lattice is geometric. Next, Denham and Suciu in \cite{DS-Massey}, and Grbi\'c and Theriault in \cite{GT-supacehtpy} exhibited coordinate arrangements with non-formal complex complements. Both teams produced non-trivial Massey products in the rational cohomology rings of the complex complement by studying moment-angle complexes which were also shown to have non-trivial Massey products by Baskakov in \cite{Baskakov-03}.

One focus of this paper is to study the question of formality for the class of subspace arrangements embedded in the braid arrangement. Towards this aim we apply the recent work of Sinha and Walter \cite{SW06} to Yuzvinsky's relative atomic complex.  This allows us to use the edge colored hypergraphs to compute certain differentials in the spectral sequence of the associated Lie coalgebra.  These differentials provide non-trivial Massey triple products.

This paper is organized as follows. Given a subspace arrangement embedded in the braid arrangement we define an edge colored hypergraph generalizing hypergraphic arrangements in Section \ref{ech}. Also,  in Section \ref{ech} we develop basic combinatorial facts translating lattice theoretic properties into the language of edge colored hypergraphs. Then in Section \ref{Poly} we prove that the characteristic polynomial of such a subspace arrangement can be calculated by a certain type of vertex coloring of the associated edge colored hypergraph. Finally, in Section \ref{complement} we study the relative atomic complex of such subspace arrangements and exhibit non-trivial Massey products for certain subspace arrangements. We finish by examining the case of $k$-equal arrangements.

{\bf Acknowledgments.} The authors would like to thank Dev Sinha and Sergey Yuzvinsky for many helpful discussions. The authors are also thankful to Takuro Abe, Hiroaki Terao, and T\'ai Huy H\'a for useful suggestions.  This work was begun when the second author was visiting Bucknell University. The second author is grateful for the gracious support of the Bucknell Mathematics Department.  The authors would like to thank the referee for helping to clarify the exposition.


\section{Definitions and combinatorics}\label{ech}

In this section we define an edge colored hypergraph from a subspace arrangement embedded in the braid arrangement. Also, given an edge colored hypergraph we define a subspace arrangement embedded in the braid arrangement. These definitions generalize that of the classical graphic arrangements (see Orlik and Terao \cite{OT}) and hypergraph arrangements (see Bj\"orner \cite{Bj-Subspaces}, Kozlov \cite{Koz-hyper}, Hultman \cite{Hult-link}). Then we study combinatorial properties of these arrangements and compute the codimensions of the subspaces using the associated edge colored hypergraph. We conclude the section with a hypergraphical interpretation of a geometric lattice together with examples.  The language we introduce for edge colored hypergraphs generalizes the same language used in the study of hypergraphs.


\subsection{Edge colored hypergraphs}

For an integer $\ell$, let $[\ell ]=\{1,\ldots, \ell\}$ and let $E$ be a finite collection of subsets of $[\ell]$, each containing at least two elements.  We say $\Ho =([\ell ],E)$ is a hypergraph. We allow hypergraphs that are not simple, which means that an edge may be contained in another edge.  See Berge \cite{Ber-hypergraphs} for a general treatment of hypergraphs.

\begin{definition}
By an \emph{edge coloring} of a hypergraph $\Ho=([\ell ],E)$ we mean a function $\C :E\to \Lambda$ for some finite set $\Lambda$. We say that a pair $(\Ho,\C)$ is a {\em edge colored hypergraph} when $\Ho$ is a hypergraph and $\C$ is a edge coloring of $\Ho$ 
\end{definition} 

Given an edge colored hypergraph $(\Ho ,\C)$ with vertices $[\ell ]$, edges $E$, and edge colors $\Lambda$ we now define a subspace arrangement $\A_{(\Ho ,\C )}$ embedded in the braid arrangement $\A_\ell$. 

\begin{definition} \label{maindef}
For any edge $e=\{i_1,i_2\ldots,i_s\}\in E$ let $$\nu(e)=\{v\in V| v_{i_1}=v_{i_2}=\cdots =v_{i_s}\}.$$ Then for each $\lambda \in \Lambda$ let $$X_\lambda = \bigcap\limits_{e\in \C^{-1}(\lambda)}\nu (e).$$ The \emph{edge colored hypergraphic arrangement} associated to $(\Ho,\C)$ is $$\A_{(\Ho,\C)}=\left\{ X_\lambda \right\} _{\lambda \in \Lambda}.$$ 
\end{definition}

Let $\A=\{X_1,\ldots ,X_n\}$ be a subspace arrangement embedded in $\A_\ell$. In order to define the edge colored hypergraph associated to $\A$ we use the isomorphism given by Orlik and Terao in the proof of Proposition 2.9 of \cite{OT} between the intersection lattice of the braid arrangement $\A_\ell$ and the partition lattice of $[\ell]$. For $r,s\in [\ell]$ with $r\neq s$ let $H_{r,s}=\{x_r-x_s=0\}$ and $H_{r,r}=V$ (notice that $H_{r,s}=\nu (\{r,s\})$ for the complete graph on the vertices $[\ell ]$, but we emphasize $H_{r,s}$ because the hypergraph has not yet been defined). For each subspace $X_i$ define the equivalence relation $\sim_i$ on $[\ell ]$ by $r\sim_is$ if and only if $X_i\subseteq H_{r,s}$. The partition of $\ell$ associated to $X_i$ is the partition defined by the equivalence classes of $\sim_i$. Denote this partition by $\pi_i =\{ B^i_1,\ldots ,B^i_{p_i}\}$. Now define an edge colored hypergraph $(\Ho_\A ,\C_\A )$ associated to $\A$, we write $(\Ho ,\C )$ when no confusion can arise. 

\begin{definition} 
The vertex set of $\Ho$ is $[\ell ]$ and the edges are $$E= \{ B^i_j | i\in \{1,\ldots ,n\} \text{, } j\in \{ 1,\ldots ,p_i\} \text{ and } |B^i_j|>1 \}.$$ The set of colors for $\Ho$ is $\Lambda =[n]$ and the color function $\C : E\to \Lambda$ is defined by $\C (B^i_j )=i$ for all $1\leq i\leq n$ and $1\leq j\leq p_i$. 
\end{definition}

\begin{remark}
We have recalled the standard definition of a subspace arrangement where there is no inclusion of subspaces. However, an edge colored hypergraphic  arrangement from Definition \ref{maindef} could very well have inclusions. 
\end{remark} 

\begin{remark}

Notice that many different edge colored hypergraphs could be associated to the same subspace arrangement embedded in $\A_\ell$. For example, let $\ell \geq3$, $e=\{1,2,3\}$, $e_1=\{1,2\}$, and $e_2=\{2,3\}$. Let $\Ho_1:=([\ell], \{e\})$, $\C_1(e)=\lambda$, $\Ho_2=([\ell],\{e_1,e_2\})$ and $\C_2(e_1)=\C_2(e_2)=\lambda$. Then the edge colored hypergraphs $(\Ho_1,\C_1)$ and $(\Ho_2,\C_2)$ are different, but their associated subspace arrangements are the same.

\end{remark}

There are many edge colored hypergraphic arrangements embedded in $\A_\ell$ which previous combinatorial data (e.g. hypergraphs) failed to capture.  For example, the subspace arrangement of all codimension $c$ subspaces, $1<c<\ell-1$, of $\A_\ell$. We end this section with a `smaller' example.

\begin{example}
Let $\ell=4$ and $(\Ho,\C)$ be the hypergraph defined by $E=\{\{1,2\},\{2,3\},\{3,4\}\}$ where the colors set is $\Lambda =\{R,B\}$ (here $R$ stands for red and $B$ for blue) and the color function is given by $\C (\{1,2\})=R$, $\C(\{2,3\})=B$, and $\C (\{3,4\})=R$. In Figure \ref{ex28} we draw the edge colored hypergraph following Berge \cite{Ber-hypergraphs}. The corresponding arrangement $\A=\{X_1,X_2\}$ is the collection of the codimension 2 space $X_1=\nu (\{1,2\})\cap \nu (\{3,4\})$ (corresponding to red) and the codimension 1 space $X_2=\nu (\{2,3\})$ (corresponding to blue). The intersection $X_1\cap X_2=\nu(\{1,2,3,4\})$ is the `diagonal'.
\begin{figure}[htbp]

\centerline {
\includegraphics[width=2.5in]{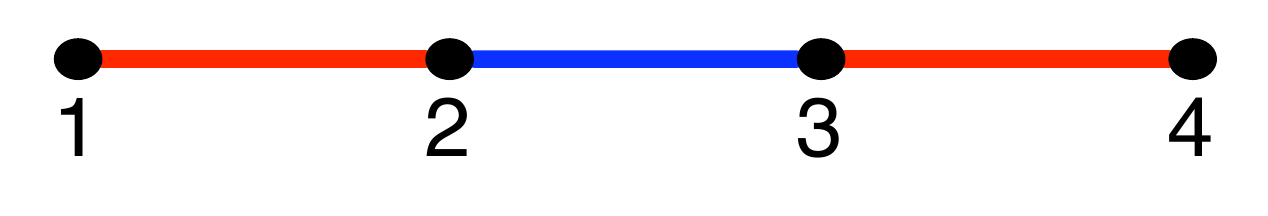}
}

\caption{An edge colored hypergraph whose arrangement could not be accounted for with only hypergraphs}\label{ex28}

\end{figure}

\end{example}


\subsection{The labeled intersection lattice}
In this section we provide the information necessary to compute the intersection lattice of an edge colored hypergraphic arrangement from the edge colored hypergraph.
Let $\A$ be an edge colored hypergraphic arrangement, $(\Ho,\C)$ its associated edge colored hypergraph, and $\Lambda$ the set of edge colors. Let $\Gamma \subseteq \Lambda$. We compute the subspace given by the intersection 
$$\bigcap\limits_{\gamma\in\Gamma}X_\gamma=\bigcap\limits_{e\in \C^{-1}(\Gamma)}\nu(e)$$
as well as its dimension in terms of $(\Ho, \C)$.
\begin{definition}
For any hypergraph $\Ho =([\ell],E)$ we say that a set of edges $\{e_1,\ldots,e_k\}\subseteq E$ is \emph{connected} if for all $i,j\in \bigcup\limits_{s=1}^ke_s$ there exists a sequence of edges $\{e_{i_1},\ldots,e_{i_t}\} \subseteq \{e_1,\ldots,e_k\}$ such that $i\in e_{i_1}$, $j\in e_{i_t}$ and for all $1\leq s\leq t-1$, $e_{i_s}\cap e_{i_{s+1}}\neq \emptyset$. Given a hypergraph $\Ho$ we call the maximal connected sets of edges {\em connected components}. \end{definition}

For any set of colors $\Gamma \subseteq \Lambda$ let $\eta =\{K_1,\ldots ,K_s\}$ be the set of connected components of the subhypergraph $\C^{-1}(\Gamma )$. Then let $\Upsilon=\{ U_1,\ldots ,U_s\}$ be the corresponding vertex sets of the connected components $\eta =\{K_1,\ldots ,K_s\}$. Then with this notation the intersection is given by 
$$\bigcap\limits_{\gamma \in \Gamma}X_\gamma=\bigcap\limits_{i=1}^s\nu (U_i).$$ 
We compute the dimensions of intersections by counting vertices.

\begin{lemma}\label{L;codim}

Let $\Ho=([\ell],E)$ be a hypergraph, $K=\{e_{i_1},\ldots,e_{i_k}\}\subseteq E$ a connected component of $\Ho$, and $\nu(e)=\{ v\in V| v_{j_1}=\cdots =v_{j_t} \text{ where } e=\{j_1,\ldots ,j_t\}\}$ be the subspace defined by $e$. Then 
$$\codim \bigcap\limits_{e\in K}\nu (e)=\left| \bigcup\limits_{e\in K}e\right| -1.$$

\end{lemma}

\proof

Let $I_K$ be the defining ideal of the subspace $ \bigcap\limits_{e\in K}\nu (e)$. Since $K$ is a connected component we know that for all $i,j\in \bigcap\limits_{e\in K}e$ the element $x_i-x_j\in I_K$. If $\bigcap\limits_{e\in K}e=\{k_1,\ldots,k_r\}$ then $I_K=(x_{k_1}-x_{k_2},x_{k_2}-x_{k_3},\ldots,x_{k_{r-2}}-x_{k_{r-1}},x_{k_{r-1}}-x_{k_r})$. Thus, $\codim  \bigcap\limits_{e\in K}\nu (e)=r-1=\left| \bigcup\limits_{e\in K}e\right| -1.$ \qed

The following Lemma is an immediate consequence of Lemma \ref{L;codim}.

\begin{lemma} \label{L;codim2}

Let $(\Ho,\C)$ be a edge colored hypergraph with edge colors $\Lambda$ and $\A_{(\Ho,\C)}$ its associated edge colored hypergraphic arrangement. For $\Gamma\subseteq \Lambda$ let $\{K_1,\ldots ,K_s\}$ be the set of all connected components of the subhypergraph induced by $\C^{-1}(\Gamma)$ and let $V_i$ be the vertex set of $K_i$ for all $i$. Then 
$$\codim\bigcap\limits_{\gamma\in\Gamma}X_{\gamma}=\sum\limits_{i=1}^s\left[ \left|V_i \right| -1\right].$$

\end{lemma}

\begin{definition}
For $\gG, \gG' \subseteq \gL$, we say that $\gG$ and $\gG'$ are {\em multiplicative} if 
\begin{gather}
\codim\bigcap\limits_{\gamma\in\Gamma}X_{\gamma} +\codim\bigcap\limits_{\gamma' \in\Gamma' }X_{\gamma' } = \codim\bigcap\limits_{\gamma \in \Gamma \cup \Gamma' }X_{\gamma}.
\end{gather} \end{definition}
This definition is important in Section \ref{complement} and Lemma \ref{L;product} will justify the term ``multiplicative."  According to Lemma \ref{L;codim2}, this condition can be checked by counting vertices. In the next example we present an edge colored hypergraph with both multiplicative and non multiplicative color sets.

\begin{example}
Let $\ell=5$ and $(\Ho,\C)$ be the hypergraph with edge set $E=\{\{1,2,3\},\{ 2,3,4\},\{3,4,5\}\}$, color set $\Lambda =\{ R,B,G\}$ (here $R$ stands for red, $B$ for blue and $G$ for green), and color function given by $\C(\{1,2,3\})=R$, $\C(\{2,3,4\})=B$, and $\C(\{3,4,5\})=G$ (see Figure \ref{ex28-2}).  Then the color sets $\{ \mathrm{red} \}$, $\{\mathrm{green}\}$ are multiplicative because the codimension of each element is 2 and the intersection is codimension 4. However, either red or green coupled with blue will not be multiplicative color sets. 
\begin{figure}[htbp]

\centerline {
\includegraphics[width=3in]{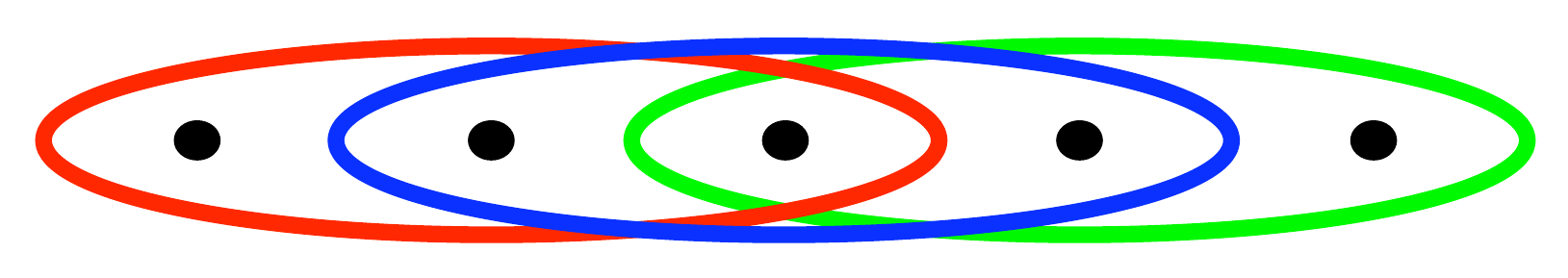}
}

\caption{An edge colored hypergraph with multiplicative color sets and non-multiplicative color sets}\label{ex28-2}
\end{figure} 

\end{example}

We can also view containment of elements in the intersection lattice through examining vertices of connected components.

\begin{definition} \label{D;refine}


Let $E$ and $E'$ be two sets of edges of a hypergraph $\Ho$ with connected components $\{ K_1,\ldots , K_s\}$ and $\{ K_1',\ldots , K_t'\}$ respectively. Let $\{ V_1,\ldots ,V_s\}$ and $\{ V_1',\ldots ,V_t'\}$ be the vertex sets of the connected components $\{ K_1,\ldots , K_s\}$ and $\{ K_1',\ldots , K_t'\}$. We say that $E$ is a \emph{refinement} of $E'$ and write $E\Subset E'$ if for each $1\leq i\leq s$ there exists a $1\leq j\leq t$ such that $V_i \subseteq V_j'$.  

\end{definition}

This is equivalent to saying that every connected component of $E$ is a subset of some connected component of $E'$. We write $E\equiv E'$ if $E\Subset E'$ and $E\Supset E'$, which is to say that the vertices of the connected components are equal.

From this definition $E\Subset E'$ if and only if 
$$\bigcap\limits_{e\in E}\nu (e) \supseteq \bigcap\limits_{f\in E'} \nu (f).$$
Also $E\equiv E'$ if and only if 
$$\bigcap\limits_{e\in E}\nu (e) = \bigcap\limits_{f\in E'} \nu (f).$$
 Now, let $(\Ho,\C)$ be an edge colored hypergraph with edge colors $\Lambda$. For $\gG,\gG' \subseteq \Lambda$  then $\C^{-1}(\Gamma )\Subset \C^{-1}(\gG' )$ if and only if $\bigcap\limits_{a\in \Gamma}X_a\supseteq \bigcap\limits_{b\in \gG'}X_b.$ For $ \Gamma,\gG' \subseteq \Lambda$ we write $\Gamma \Subset \gG'$ if $\C^{-1}(\Gamma )\Subset \C^{-1}(\gG' )$. For the remainder of this paper we assume that for all $\lambda,\lambda'\in \Lambda$, $\lambda\not\Subset \lambda'$.  This assumption ensures that there is no inclusion of subspaces in the arrangement. 

 We will use the following elementary lemma in the next section.

\begin{lemma}\label{subcolors}

Let $(\Ho,\C)$ be an edge colored hypergraph with edge colors $\Lambda$. Fix $\Gamma_1,\Gamma_2\subseteq \Lambda$. Let $\Psi_1$ and $\Psi_2$ be maximal color sets such that $\Psi_1\equiv \Gamma_1$ and $\Psi_2\equiv \Gamma_2$. Then $\C^{-1}(\Gamma_1) \Subset \C^{-1}(\Gamma_2)$ if and only if $\Psi_1\subseteq \Psi_2$.

\end{lemma}


\subsection{Geometric edge colored hypergraphs}

Let $(\Ho,\C)$ be an edge colored hypergraph with edge colors $\Lambda$ and let $\A_{(\Ho,\C)}$ be the associated subspace arrangement. Let $\Gamma_1,\Gamma_2\subseteq \Lambda$ and $X_1=\bigcap\limits_{\gamma \in \Gamma_1}X_\gamma =\bigcap\limits_{e\in \C^{-1}(\Gamma_1 )}\nu (e)$ and $X_2=\bigcap\limits_{\gamma \in \Gamma_2}X_\gamma =\bigcap\limits_{e\in \C^{-1}(\Gamma_2 )}\nu (e)$ be the respective elements of the intersection lattice $L(\A_{(\Ho,\C)})$. Then the least upper bound or join of $X_1$ and $X_2$ is $$X_1\vee X_2=X_1 \cap X_2 = \bigcap\limits_{\gamma \in \Gamma_1 \cup \Gamma_2}X_\gamma=\bigcap\limits_{e\in \C^{-1}(\Gamma_1\cup \Gamma_2)}\nu (e).$$ Computing greatest lower bounds or meets in intersection lattices of subspaces is in general not tractable. However, in this setting we can view a meet of two elements of the intersection lattice in terms of edge colors. The next lemma follows from Lemma \ref{subcolors}.

\begin{lemma}

Let $X_1,X_2\in L(\A_{(\Ho,\C)})$ and let $\Psi_1,\Psi_2\subseteq \Lambda$ be maximal color sets such that $X_i=\bigcap\limits_{\gamma \in \Psi_i}X_\gamma$ for $i=1,2$. If $\Psi_1\cap \Psi_2\neq \emptyset$ then the meet of $X_1$ and $X_2$ is  $$X_1\wedge X_2 =\bigcap\limits_{e\in \C^{-1}(\Psi_1\cap \Psi_2)}\nu (e).$$ If $\Psi_1\cap \Psi_2= \emptyset$ then $X_1\wedge X_2=V$.

\end{lemma}

As in this Lemma we define the join of two color sets $\Gamma_1$ and $\Gamma_2$ as $\Gamma_1 \wedge \Gamma_2:=\Psi_1 \cap \Psi_2$ where $\Psi_1$ and $\Psi_2$ are maximal color sets such that $\Psi_1\equiv \Gamma_1$ and $\Psi_2\equiv \Gamma_2$. This is an abuse of notation because we are considering the empty set as a color with no edges.

\begin{definition}\label{covers}

Let $(\Ho,\C)$ be an edge colored hypergraph with edge colors $\Lambda$ and let $\Gamma_1,\Gamma_2\subseteq \Lambda$. We say that $\Gamma_1$ \emph{covers} $\Gamma_2$ if $\Gamma_1\Supset \Gamma_2$ and there does not exist a set of colors $\Psi$ that satisfies $\Psi \not\equiv \Gamma_i$ for $i=1,2$ and $\Gamma_1\Supset \Psi \Supset \Gamma_2$. 

\end{definition}

Now we give conditions on an edge colored hypergraph so that its corresponding subspace arrangement has a geometric intersection lattice.

\begin{definition}\label{geohyp}

We say that an edge colored hypergraph $(\Ho,\C)$ is geometric if whenever two color sets $\Gamma_1$ and $\Gamma_2$ cover $\Gamma_1\wedge \Gamma_2$ then $\Gamma_1 \cup \Gamma_2$ covers $\Gamma_1$ and $\Gamma_2$.

\end{definition}

This definition is different from the definition of a `geometric hypergraph' found in Helgason \cite{Helgason}.  The next proposition follows immediately since the intersection lattice of any subspace arrangement is atomic and Definition \ref{geohyp} is exactly the semimodular property.

\begin{proposition}

Let $(\Ho ,\C)$ be an edge colored hypergraph. The intersection lattice $L(\A_{(\Ho ,\C)})$ is geometric if and only if $(\Ho ,\C)$ is a geometric hypergraph.

\end{proposition}

We conclude this section with examples.

\begin{example}

Let $\Ho = ( [4], \{a,b,c\} )$ where $a=\{1,2,3\}$, $b=\{3,4\}$, and $c=\{2,4\}$, and let each edge have it's own color.  In Figure \ref{smalldude} we draw the hypergraph as well as the intersection lattice of the corresponding arrangement.
\begin{figure}[htbp]

\centerline {
\includegraphics[width=2.5in]{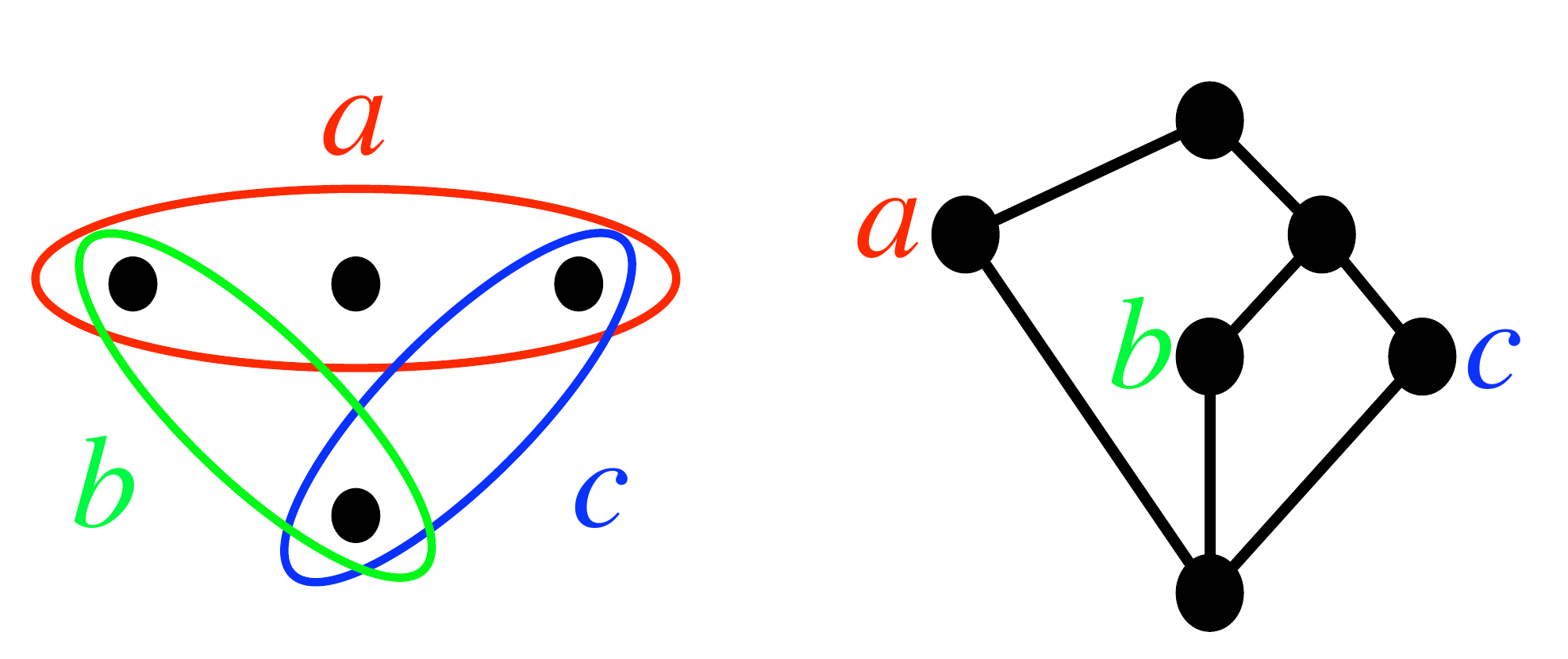}
}

\caption{On the left is the smallest hypergraph that is not geometric with the corresponding intersection lattice on the right.}
\label{smalldude}
\end{figure} 
Notice that this intersection lattice is the smallest non-geometric atomic lattice. 

\end{example}

The next examples illustrate that the extremal codimensional cases have geometric hypergraphs.

\begin{example}[Hyperplane arrangements]

Let the hypergraph $(\Ho ,\C)$ have edges with exactly two vertices so that it is a graph and let each edge have its own color. Then the arrangement is a graphic hyperplane arrangement and all intersection lattices of hyperplane arrangements are geometric. Hence, graphs with a different edge color for each edge are geometric.

\end{example}

\begin{example}[Line arrangements]

Now, let $(\Ho, \C)$ with vertex set $[\ell ]$ be a hypergraph where every edge has $\ell -1$ vertices. Then $(\Ho,\C)$ is geometric because the intersection lattice of the associated arrangement only consists of the origin, the vector space, and the lines. 

\end{example}


\section{Characteristic Polynomials}\label{Poly}

In this section we construct a generalized chromatic polynomial from an edge colored hypergraph $(\Ho ,\C)$ and show that this polynomial is equal to the characteristic polynomial of the associated edge colored hypergraphic arrangement $\A_{(\Ho,\C)}$. First let us review some classical invariants for arbitrary subspace arrangements. Given an arbitrary subspace arrangement $\A$ the M\"obius function on $L(\A)$ written as $\mu :L(\A)\to \ZZ$ is defined recursively by $\mu (V)=1$ and for all $X\in L(\A)$
$$\mu (X)=-\sum\limits_{X\subsetneq Y}\mu (Y)$$
(see Rota \cite{Rota-mobius} or Stanley \cite{Stan-comb}).
The characteristic polynomial of $\A$ is a polynomial in one variable defined by 
$$\chi (\A,t)=\sum\limits_{X\in L(\A)}\mu (X)t^{\dim X}.$$
The study of $\chi(\A,t)$ is a celebrated area of mathematics and has connected many fields ranging from combinatorics to topology to complexity theory, see for example Athanasiadis \cite{Ath}, Bj\"orner \cite{Bj-Subspaces}, Orlik and Terao \cite{OT}, Sagan \cite{Sag-charfactor}, and Zaslavsky \cite{Zas-char}. However, as noted by Bj\"orner in \cite{Bj-Subspaces}, for higher codimensional subspace arrangements this polynomial does not necessarily contain topological information of the arrangements real or complex complement. Thus, the results in this section are combinatorial and are not necessarily related to the Poincar\'e polynomial of the complement.  

Now let us review the foundational results of Blass and Sagan from \cite{BS-Charac}. They examine the characteristic polynomial of an subspace arrangement embedded in the Coxeter arrangement of type $\text{B}_\ell$. Let $[-s,s]^\ell$ be the set of all integer valued points in the $\ell$-cube of $\ZZ^\ell$ where $[-s,s]=\{-s,-s+1,\ldots, s-1,s\}$. Let $ \bigcup \A= \bigcup\limits_{X\in \A}X$ considered as a subset of $\ZZ^\ell$. In \cite{BS-Charac} Blass and Sagan prove the following theorem.

\begin{theorem}[Blass, Sagan]\label{BS1}
If $\A$ is embedded in $B_\ell$ then $$\chi (\A, t)=\# ([-s,s]^\ell\backslash \bigcup \A )$$ where $t=2s+1$.
\end{theorem}

This is a generalization of Zaslavsky's results in \cite{Zas-Chrom} and \cite{Zas-Signedcolor} concerning the relationship of the characteristic polynomial of a subspace arrangement $\A$ embedded in $B_\ell$, and the chromatic polynomial of a signed graph associated to $\A$. In this section we connect these two points of view for subspace arrangements embedded in $\A_\ell$.

In \cite{Zas-Chrom} and \cite{Zas-Signedcolor}, Zaslavsky constructs a signed graph for hyperplane arrangements (not general subspace arrangements) which are embedded in $B_\ell$. The signed graph is the edge colored hypergraph we defined above where all edges are sets with only two elements and each edge has its own color.  In addition there is a sign for each edge which indicates that the hyperplane is defined by $x_i-x_j=0$ or $x_i+x_j=0$, and there are `half' edges that correspond to hyperplanes defined by $x_i=0$.  Given such a graph $G_\A$, Zaslavsky defines a chromatic polynomial $\chi (G_\A,t)$ of $G$ and he proves the following theorem. 

\begin{theorem}[Zaslavsky]\label{Zas1}

Let $\A$ be a hyperplane arrangement embedded in $L(B_\ell )$. Then $$\chi (G_\A ,t)=\chi (\A,t).$$

\end{theorem}

We generalize Zaslavsky's ideas of vertex graph coloring to vertex coloring of an edge colored hypergraph. Then we use Theorem \ref{BS1} and the ideas from Theorem \ref{Zas1} to show that this is equal to the characteristic polynomial.

Fix an edge colored hypergraphic arrangement $\A_{(\Ho,\C)}$ in $V=\CC^\ell$ where $(\Ho,\C)$ is the associated edge colored hypergraph. A vertex coloring of $(\Ho,\C)$ is a map $c:[\ell] \to \gX$ where $\gX$ is a finite set of $t$ colors.

\begin{definition}\label{proper}

We say that a vertex coloring $c$ of $(\Ho,\C)$ is \emph{proper} if for every $\lambda \in \Lambda$ there exists a connected component $K\in \C^{-1}(\lambda )$ and some $i,j\in K$ with $c(i)\neq c(j)$.

\end{definition}

\begin{definition}\label{chromatic}

Let $(\Ho,\C)$ be an edge colored hypergraph. The \textit{chromatic polynomial} of $(\Ho,\C)$ is a polynomial in the variable $t$ defined by $$\chi (\Ho,\C ,t)=\# (\text{proper vertex colorings of } (\Ho,\C) \text{ with } t \text{ colors}).$$

\end{definition}

\begin{remark}
One can also define a signed edge colored hypergraph with half edges generalizing the results of Blass and Sagan \cite{BS-Charac} and Zaslavsky \cite{Zas-Chrom} and \cite{Zas-Signedcolor}.  
\end{remark}

To justify that this is a polynomial we define deletion and contraction of an edge colored hypergraph.  

\begin{definition}

Let $(\Ho , \C) = ( ( V , E) , \C)$ be a an edge colored hypergraph with colors $\gL$.  Fix $\gl \in \gL$ and let $\{K_1 , \ldots , K_s\}$ be the connected components of $\C^{-1}(\lambda )$.  We define the {\em deletion} of $(\Ho,\C)$ by $\gl$ to be 
$$(\Ho' , \C') = ( ( V , E \backslash \C^{-1} ( \gl )) , \C|_{E \backslash \C^{-1} ( \gl )} ).$$
We define the {\em contraction} of $(\Ho , \C)$ along $\gl$ to be 
$$( \Ho'' , \C'') = ( (V'' , E \backslash \C^{-1}(\gl)) , \C|_{E \backslash \C^{-1} ( \gl )}),$$
where $V''$ is vertex set $V$ with the vertices of $K_i$ identified for each $i$.

\end{definition}

The next proposition follows from a standard deletion-contraction argument.

\begin{proposition}\label{dele-contrac} 
Let $(\Ho, \C)$ be an edge colored hypergraph with $( \Ho' , \C')$ and $( \Ho'' , \C'')$ as defined above then
$$\gc(  ( \Ho , \C) , t ) = \gc(  ( \Ho' , \C') , t ) - \gc(  ( \Ho'' , \C'') , t ).$$

\end{proposition}

The previous proposition provides the inductive step for an induction on the number of colors to show that the chromatic polynomial for an edge colored hypergraph is indeed a polynomial.  The base case for the induction is the edge colored hypergraph with no edges which has chromatic polynomial $t^{\ell}$.

Now, we state the main theorem of this section. It can be proved by using a standard deletion-contraction argument and Proposition \ref{dele-contrac}. We choose to prove it by using Theorem \ref{BS1} of Blass and Sagan because each point in the integer complement gives an explicit vertex coloring of the arrangements associated edge colored hypergraph.

\begin{theorem}\label{character}

If $(\Ho,\C)$ is an edge colored hypergraph and $\A_{(\Ho,\C)}$ is the associated edge colored hypergraphic arrangement then $$\chi (\A_{(\Ho,\C)},t)= \chi (\Ho,\C ,t).$$

\end{theorem}

\begin{proof}

Let $\Lambda$ be the set of edge colors for $(\Ho,\C)$. By Theorem \ref{BS1} $$\chi (\A_{(\Ho,\C)},t)=\# ([-s,s]^\ell\backslash \bigcup \A_{(\Ho,\C)} ).$$ Now, we construct a correspondence between an element of $[-s,s]^\ell\backslash \bigcup \A_{(\Ho,\C)}$ and a proper vertex coloring of the edge colored hypergraph $(\Ho,\C)$.

Let $(m_1,\ldots ,m_\ell )\in [-s,s]^\ell\backslash \bigcup \A$ and define a vertex coloring of $(\Ho,\C)$ by the function $c:[\ell]\to [-s,s]$ defined by $c(i)=m_i$. Fix an arbitrary $\lambda\in\Lambda$. Then the subspace corresponding to $\lambda$ is $$X_\lambda=\bigcap\limits_{e\in \C^{-1}(\lambda)}\nu (e)\ \ .$$ Since $(m_1,\ldots ,m_\ell )\notin \bigcup \A_{(\Ho,\C)}$ then there exists a connected component $K$ of $ \C^{-1}(\lambda)$ such that $(m_1,\ldots,m_\ell) \notin \nu (K)$. Suppose that $\nu (K)=\{v\in V| v_{k_1}=v_{k_2}=\cdots =v_{k_s}\}$. Since $(m_1,\ldots,m_\ell) \notin \nu (K)$ there exists $i,j\in\{k_1,\ldots,k_s\}$ such that $c(i)=m_i\neq m_j=c(j)$. Therefore, $c$ is a proper vertex coloring of $(\Ho,\C)$.

Now, let $c:[\ell]\to [-s,s]$ be a proper vertex coloring of the edge colored hypergraph $(\Ho,\C)$. Define an element of $[-s,s]^\ell$ by the vector $(c(1),c(2),\ldots,c(\ell))$. Then for every color $\lambda\in \Lambda$ there exists a connected component $K$ of $ \C^{-1}(\gl)$ with $i,j\in K$ such that $c(i)\neq c(j)$. Thus, the vector $(c(1),c(2),\ldots,c(\ell))$ is not in the subspace $\nu (K)$ and hence not in the subspace $X_\lambda$. Since $\lambda$ was chosen arbitrarily we know that the vector $(c(1),c(2),\ldots,c(\ell))$ can not be an element of $\bigcup \A_{(\Ho,\C)}$.  \end{proof}



\section{The complex complement}\label{complement}

In this section we turn to the topology of the complex complement of a subspace arrangement embedded in the braid arrangement. First we study the relative atomic complex of the intersection lattice defined by Yuzvinsky in \cite{Yuz-SmalRat}.  We then review the Lie coalgebra of a differential graded algebra and it's spectral sequence, developed by Sinha and Walter in \cite{SW06}. Next we apply this Lie coalgebra to Yuzvinsky's relative atomic complex in order to produce subspace arrangements with non-trivial Massey products. We conclude by examining the special case of $k$-equal arrangements


\subsection{Yuzvinsky's relative atomic complex}

The relative atomic complex is a rational model for subspace arrangements, defined by Yuzvinsky in \cite{Yuz-SmalRat} and used by Feichtner and Yuzvinsky in \cite{FY-Formal} to prove that the complement of a subspace arrangement with a geometric intersection lattice is formal.

Let $\A=\{ X_1,\ldots ,X_n\}$ be a subspace arrangement and fix an order on these elements, $X_1<X_2<\cdots <X_n$.  We associate the integer $s$ with the subspace $X_s$. Then we will use $\gs= \{i_1, \ldots, i_k\}\subseteq \{1,\ldots ,n\}$ to denote a subset of atoms in the intersection lattice of $\A$ such that $1\leq i_1< i_2 < \cdots < i_k$. Let $D_{\mc{A}}$ be the differential graded algebra over $\QQ$ with a generator $a_{\gs}$ in degree $2 \, \text{codim} \bigvee \gs - | \gs | $ for each $\gs$.  The differential is defined by
\begin{gather} \label{D;d}
d a_{\gs} = \sum_{j: \bigvee \gs \bs i_j = \bigvee \gs} (-1)^j a_{\gs \bs i_j},
\end{gather}
and the product structure is defined by
\begin{gather} \label{D;product}
a_{\gs} a_{\gt} = 
\begin{cases}
(-1)^{\ge(\gs,\gt)} a_{\gs \cup \gt} & \codim \bigvee \gs + \codim \bigvee \gt = \codim \bigvee (\gs \cup \gt) \\
0 & \text{otherwise},
\end{cases}
\end{gather}
where $\ge(\gs,\gt)$ is the sign associated with the permutation that re-orders the linearly ordered $\gs \cup \gt$ so that the elements of $\gt$ come after that of $\gs$. 

Feichtner and Yuzvinsky in \cite{FY-Formal} prove the next theorem by showing that this relative atomic complex $D_\A$ is quasi-isomorphic to the rational models of the complex complement of $\A$ developed by De Concini and Procesi in \cite{DCP95}.

\begin{theorem}[Feichtner, Yuzvinsky]

Let $\A$ be a subspace arrangement in $V\cong \CC^\ell$. Then the relative atomic complex $D_\A$ is a rational model for the complex complement $M(\A)$.

\end{theorem}

We return to the case of an edge colored hypergraphic arrangement with hypergraph $(\Ho,\C)$ and edge colors $\Lambda$. Choose an order on the colors $\Lambda =\{ \lambda_1,\ldots ,\lambda_n\}$, $\lambda_1<\cdots <\lambda_n$, hence ordering the atoms of the intersection lattice $L(\A_{(\Ho,\C)})$. For an ordered subset of colors $\gG =\{\lambda_{i_1},\ldots ,\lambda_{i_k} \} \subseteq \gL$, we abbreviate the algebra generator $a_{i_1i_2\cdots i_k}$ by $a_{\gG}$.  The next lemma follows directly from Definition \ref{D;refine} and Equation (\ref{D;d}).

\begin{lemma} \label{L;d}

Let $\gG \subseteq \gL$, then $da_{\gG}  = \sum\limits (-1)^{j} a_{\gG \bs \gl}$, where the sum is over all $\gl \in \gG$ such that $\gG \bs \gl \Supset \gl $ and $\gl$ is the $j^{\text{th}}$ element of $\gG$.

\end{lemma}

By Lemma \ref{L;codim2} and Equation (\ref{D;product}) we have the following lemma.

\begin{lemma} \label{L;product}

Let $\gG, \gG' \subset \gL$.  The product $a_{\gG} a_{\gG'} = (-1)^{\ge(\gG, \gG')} a_{\gG \cup \gG'}$ if and only if $\gG$ and $\gG'$ are multiplicative.

\end{lemma}


\subsection{Sinha and Walter's Lie coalgebras}

The machinery developed by Sinha and Walter in \cite{SW06} provides a computationally friendly way to prove the existence of certain non-vanishing Massey products.  They construct a pair of Quillen adjoint functors from the category of rational commutative differential graded algebras (DGA) to the category of rational differential graded Lie coalgebras (DGE).  We apply their functor $\mc{E} : \text{DGA} \to \text{DGE}$ to Yuzvinsky's relative atomic complexes described above.

Let Graphs(\(n\)) denote the vector space spanned by acyclic directed graphs with \(n\) labeled vertices.
\begin{definition} \label{D:Eil}
Let \(W\) be a graded vector space.  The co-free Lie co-algebra on \(W\) is
\begin{gather}
\bb{E}(W) := \bigoplus_n  \left( ( \text{Graphs}(n)  / \sim ) \otimes_{\gS_n} W^{\otimes n}  \right) .
\end{gather}
Here \(\gS_n\) is the symmetric group on \(n\) letters which acts diagonally, and the equivalence relation on Graphs(\(n\)) is generated by 

\begin{align*}
\text{(arrow-reversing)}\qquad & \qquad
\begin{xy}                           
(0,-2)*+UR{\scriptstyle a}="a",    
(3,3)*+UR{\scriptstyle b}="b",     
"a";"b"**\dir{-}?>*\dir{>},         
(1.5,-5),{\ar@{. }@(l,l)(1.5,6)},
?!{"a";"a"+/va(210)/}="a1",
?!{"a";"a"+/va(240)/}="a2",
?!{"a";"a"+/va(270)/}="a3",
"a";"a1"**\dir{-},  "a";"a2"**\dir{-},  "a";"a3"**\dir{-},
(1.5,6),{\ar@{. }@(r,r)(1.5,-5)},
?!{"b";"b"+/va(30)/}="b1",
?!{"b";"b"+/va(40)/}="b2",
?!{"b";"b"+/va(50)/}="b3",
"b";"b1"**\dir{-},  "b";"b2"**\dir{-},  "b";"b3"**\dir{-},
\end{xy}\ =\ \ -  
\begin{xy}                           
(0,-2)*+UR{\scriptstyle a}="a",    
(3,3)*+UR{\scriptstyle b}="b",     
"a";"b"**\dir{-}?<*\dir{<},         
(1.5,-5),{\ar@{. }@(l,l)(1.5,6)},
?!{"a";"a"+/va(210)/}="a1",
?!{"a";"a"+/va(240)/}="a2",
?!{"a";"a"+/va(270)/}="a3",
"a";"a1"**\dir{-},  "a";"a2"**\dir{-},  "a";"a3"**\dir{-},
(1.5,6),{\ar@{. }@(r,r)(1.5,-5)},
?!{"b";"b"+/va(30)/}="b1",
?!{"b";"b"+/va(40)/}="b2",
?!{"b";"b"+/va(50)/}="b3",
"b";"b1"**\dir{-},  "b";"b2"**\dir{-},  "b";"b3"**\dir{-},
\end{xy} \\
\text{(Arnold)}\qquad & \qquad
\begin{xy}                           
(0,-2)*+UR{\scriptstyle a}="a",    
(3,3)*+UR{\scriptstyle b}="b",   
(6,-2)*+UR{\scriptstyle c}="c",   
"a";"b"**\dir{-}?>*\dir{>},         
"b";"c"**\dir{-}?>*\dir{>},         
(3,-5),{\ar@{. }@(l,l)(3,6)},
?!{"a";"a"+/va(210)/}="a1",
?!{"a";"a"+/va(240)/}="a2",
?!{"a";"a"+/va(270)/}="a3",
?!{"b";"b"+/va(120)/}="b1",
"a";"a1"**\dir{-},  "a";"a2"**\dir{-},  "a";"a3"**\dir{-},
"b";"b1"**\dir{-}, "b";(3,6)**\dir{-},
(3,-5),{\ar@{. }@(r,r)(3,6)},
?!{"c";"c"+/va(-90)/}="c1",
?!{"c";"c"+/va(-60)/}="c2",
?!{"c";"c"+/va(-30)/}="c3",
?!{"b";"b"+/va(60)/}="b3",
"c";"c1"**\dir{-},  "c";"c2"**\dir{-},  "c";"c3"**\dir{-},
"b";"b3"**\dir{-}, 
\end{xy}\ + \                             
\begin{xy}                           
(0,-2)*+UR{\scriptstyle a}="a",    
(3,3)*+UR{\scriptstyle b}="b",   
(6,-2)*+UR{\scriptstyle c}="c",    
"b";"c"**\dir{-}?>*\dir{>},         
"c";"a"**\dir{-}?>*\dir{>},          
(3,-5),{\ar@{. }@(l,l)(3,6)},
?!{"a";"a"+/va(210)/}="a1",
?!{"a";"a"+/va(240)/}="a2",
?!{"a";"a"+/va(270)/}="a3",
?!{"b";"b"+/va(120)/}="b1",
"a";"a1"**\dir{-},  "a";"a2"**\dir{-},  "a";"a3"**\dir{-},
"b";"b1"**\dir{-}, "b";(3,6)**\dir{-},
(3,-5),{\ar@{. }@(r,r)(3,6)},
?!{"c";"c"+/va(-90)/}="c1",
?!{"c";"c"+/va(-60)/}="c2",
?!{"c";"c"+/va(-30)/}="c3",
?!{"b";"b"+/va(60)/}="b3",
"c";"c1"**\dir{-},  "c";"c2"**\dir{-},  "c";"c3"**\dir{-},
"b";"b3"**\dir{-}, 
\end{xy}\ + \                              
\begin{xy}                           
(0,-2)*+UR{\scriptstyle a}="a",    
(3,3)*+UR{\scriptstyle b}="b",   
(6,-2)*+UR{\scriptstyle c}="c",    
"a";"b"**\dir{-}?>*\dir{>},         
"c";"a"**\dir{-}?>*\dir{>},          
(3,-5),{\ar@{. }@(l,l)(3,6)},
?!{"a";"a"+/va(210)/}="a1",
?!{"a";"a"+/va(240)/}="a2",
?!{"a";"a"+/va(270)/}="a3",
?!{"b";"b"+/va(120)/}="b1",
"a";"a1"**\dir{-},  "a";"a2"**\dir{-},  "a";"a3"**\dir{-},
"b";"b1"**\dir{-}, "b";(3,6)**\dir{-},
(3,-5),{\ar@{. }@(r,r)(3,6)},
?!{"c";"c"+/va(-90)/}="c1",
?!{"c";"c"+/va(-60)/}="c2",
?!{"c";"c"+/va(-30)/}="c3",
?!{"b";"b"+/va(60)/}="b3",
"c";"c1"**\dir{-},  "c";"c2"**\dir{-},  "c";"c3"**\dir{-},
"b";"b3"**\dir{-}, 
\end{xy}\ =\ 0                             
\end{align*}

\end{definition}

If \(W\) is a DGA there are two differentials on \( \bb{E}(W)\), the first of which is denoted \(d_{W}\) and is the canonical extension of the differential on \(W\).  The second differential, denoted \(d_{\gm}\), is defined using the combinatorics of graphs.  We begin by defining a map \(\gm_e\), for an edge \(e\) of a graph. 

\begin{definition}
Let $g$ be a homogeneous element of $\bb{E}(W)$.  For every edge $e$ of the underlying graph of $g$ we may construct a new ordered labeled graph $\mu_e(g)$
as follows.

Pick a representative of $g$ up to the $\Sigma_n$ action
in which edge $e$ goes from 
vertex number $1$ to vertex number $2$, with the first two entries
of the associated tensor being $a$ and $b$.
Contract the edge from $1$ to $2$ in this representative to  a vertex 
which is then given the number
$1$ and first entry in the tensor of  $(-1)^{|a|} (a b)$ where $|a|$ denotes the degree of $a$.
In this operation, the ordering of all other vertices in the graph is 
shifted down by one to make up for the now missing 2.
$$\mu_e:
\overtie{
\begin{xy}                           
(0,-2)*+UR{\scriptstyle 1}="a",    
(3,3)*+UR{\scriptstyle 2}="b",     
"a";"b"**\dir{-}?>*\dir{>} ?(.4)*!RD{\scriptstyle e\,},         
(1.5,-5),{\ar@{. }@(l,l)(1.5,6)},
?!{"a";"a"+/va(210)/}="a1",
?!{"a";"a"+/va(240)/}="a2",
"a";"a1"**\dir{-},  "a";"a2"**\dir{-},  
(1.5,6),{\ar@{. }@(r,r)(1.5,-5)},
?!{"b";"b"+/va(30)/}="b2",
?!{"b";"b"+/va(60)/}="b3",
"b";"b2"**\dir{-},  "b";"b3"**\dir{-},
\end{xy}
}{ a\otimes  b\otimes\cdots}\ \longmapsto\ 
(-1)^{|a|}\!\!\!\!\!\!\!\overtie{
\begin{xy}
(0,1)*+UR{\scriptstyle 1}="a",
(0,-3),{\ar@{. }@(l,l)(0,4)},
?!{"a";"a"+/va(210)/}="a1",
?!{"a";"a"+/va(240)/}="a2",
"a";"a1"**\dir{-},  "a";"a2"**\dir{-},  
(0,4),{\ar@{. }@(r,r)(0,-3)},
?!{"a";"a"+/va(30)/}="b2",
?!{"a";"a"+/va(60)/}="b3",
"a";"b2"**\dir{-},  "a";"b3"**\dir{-},
\end{xy}}{ (ab)\otimes \cdots}
$$
\end{definition}
We then define the differential \(d_{\gm}\) as follows
\begin{gather}
d_\mu (g) = 
\sum_{e} \mu_e (g).
\end{gather}
	
Since \(\bb{E}(W)\) is a bi-complex, its homology with respect to the total differential can be computed from the spectral sequence given by filtering it by columns.  In this filtration the \(d_0\) is the differential induced by \(d_W\).  We refer to this spectral sequence as the Sinha-Walter spectral sequence of $W$.

The bi-complex \(\bb{E}(W)\) also has a Lie co-bracket defined combinatorially by removing edges.  As we do not make use of it here we omit the definition.

For ``long $n$-graphs" we let $a_1|a_2|\cdots |a_n$ denote the equivalence class of
$$  {\begin{xy}
(0,-2)*+UR{\scriptstyle 1}="1",
(3.5,3)*+UR{\scriptstyle 2}="2",
(7,-2)*+UR{\scriptstyle 3}="3",
(10.5,3)*+UR{\scriptstyle 4}="4",
(14,-2)*+UR{\scriptstyle n-1}="5",
(17.5,3)*+UR{\scriptstyle n}="6",
"1";"2"**\dir{-}?>*\dir{>},
"2";"3"**\dir{-}?>*\dir{>},
"3";"4"**\dir{-}?>*\dir{>},
"4";"5"**\dir{.},
"5";"6"**\dir{-}?>*\dir{>},
\end{xy}{\textstyle \bigotimes\,} {a_1\otimes a_2\otimes \cdots \otimes a_n}} $$ in \(\bb{E}(W)\).

We use the following proposition to compute non-trivial Massey products.  Let $W$ be a DGA and let $E_r^{p,q}$ be the $r^{\textrm{th}}$ page of the Sinha-Walter spectral sequence of $W$, with differential $d_r$.  

\begin{proposition} \label{P;MPs} 

If the $r^{th}$ order Massey product $\langle [a_1] , [a_2], \ldots , [a_r] \rangle$ is defined in $H^*(W)$ and $[ a_1 |  a_2 | \ldots | a_r ] $ is non-zero then $[ a_1 |  a_2 | \ldots | a_r ] $ survives to the $(r-1)^{st}$ page and
$$d_{r-1}( a_1|a_2| \ldots | a_r) \in \pm \big[ \langle [a_1] , [a_2], \ldots , [a_r] \rangle  \big].$$

\end{proposition}

The proof follows from the description of the Sinha-Walter spectral sequence by the same argument as given by McCleary \cite{McCleary} in Theorem 8.31.

Another reason to use the Sinha-Walter spectral sequence is that it allows us to compute rational homotopy groups.  Denote $\Hom(\pi_*(X) , \QQ)$ by $\pi^*(X)$.

\begin{proposition} \label{P;pi} (Sinha, Walter)

Let $W$ be a rational model for the finite complex $X$.  The Sinha-Walter spectral sequence for $W$ converges to $\pi^*(X)$.

\end{proposition}

\begin{remark}
The Lie coalgebra $\mc{E}(A)$ is isomorphic to the Harrison complex of the commutative algebra $A$ equipped with the additional structure of a Lie coalgebra, see \cite{SW06}.  In the results that follow it suffices to use the standard Harrison complex, though we suspect that the combinatorics of the graphs used by Sinha and Walter as well as the Lie coalgebra structure will be useful for our goal of understanding the rational homotopy type of edge colored hypergraphic arrangements.
\end{remark}


\subsection{Lie coalgebras of the relative atomic complex}

In this section we show that the complement of an edge colored hypergraphic arrangement $\A_{(\Ho,\C)}$ admits non-trivial Massey products if certain conditions on the edge colored hypergraph $(\Ho,\C)$ are satisfied.

\begin{definition}\label{colorsystem}

Let $(\Ho,\C)$ be an edge colored hypergraph with edge colors $\Lambda$. Let $\lambda_1,\lambda_2,\lambda_3\in \Lambda$. We call $(\lambda_1,\lambda_2,\lambda_3)$ a \emph{Massey color system} if the pairs $\lambda_1$,  $\lambda_2$ and $\{\lambda_1,\lambda_2\}$, $\lambda_3$ are multiplicative and there exists $\lambda_4,\lambda_5\in \Lambda$ such that 
\begin{align}
\{\lambda_1,\lambda_2\} & \Supset \lambda_4 &
\{\lambda_2,\lambda_4\}& \not\Supset \lambda_1 &
\{\lambda_1,\lambda_4\}& \not\Supset \lambda_2   		\label{E;124} \\
\{\lambda_2,\lambda_3\} & \Supset \lambda_5 &
\{\lambda_3,\lambda_5\} & \not\Supset \lambda_2 & 
\{\lambda_2,\lambda_5\} & \not\Supset \lambda_3. 		\label{E;235}
\end{align} We call $\gl_4$ and $\gl_5$ \emph{embedded colors} for the triple $\gl_1, \gl_2,\gl_3$.

\end{definition}

\begin{theorem} \label{T;MP3}

Let $\A$ be a subspace arrangement embedded in $\A_{\ell}$ with corresponding edge colored hypergraph $(\Ho , \C)$ and edge colors $\gL$.  Assume that $( \gl_1, \gl_2, \gl_3)$ is a Massey color system with embedded colors $\gl_4$ and $\gl_5$ as in Definition \ref{colorsystem}.  Choose a linear order $\ll$ on the colors so that $\gl_1 \ll \gl_2 \ll \gl_3 \ll \gl_4 \ll \gl_5$.  Then in the Sinha-Walter spectral sequence for $D_{\A}$ 
$$d_2(a_{\gl_1} | a_{\gl_2} | a_{\gl_3} ) = \left[ a_{ \{\gl_1, \gl_2 , \gl_3, \gl_4\}} + a_{\{\gl_1, \gl_2 , \gl_3, \gl_5\}} \right] .$$

\end{theorem}

Figure \ref{d2} illustrates the proof of Theorem \ref{T;MP3}, the solid arrow represents the differential $d_2$.

\begin{figure}[htbp] \label{d2}
\scalebox{.75}{
$\xymatrix@R=.25cm@C=.25cm{
&& \ar@{-}[dddddd] && \ar@{-}[dddddd] && \ar@{-}[dddddd] && 
\\
& 
a_{\gl_1} \, | \, a_{\gl_2} \, | \, a_{\gl_3 } \ar@{-->}[rr]  \ar[ddrrrr] && 
a_{\gl_1} a_{\gl_2} \, | \, a_{\gl_3}  + \; a_{\gl_1} \, | \, a_{\gl_2} a_{\gl_3}  &&
&& 
\\
& 
&&
&&
&& 
\\
&
&&
\ar@{-->}[uu] \ar@{-->}[rr]
a_{\{\gl_1, \gl_2, \gl_4\}} \, | \, a_{\gl_3} + a_{\gl_1} \, | \, a_{\{\gl_2, \gl_3, \gl_5 \}}
&&
a_{\{ \gl_1 , \gl_2 , \gl_3 , \gl_4 \}} + a_{\{ \gl_1 , \gl_2 , \gl_3 , \gl_5 \}}
&& 
\\
\ar@{-}[rrrrrrrr] &&&&&&&& \\
& 2 && 1 && 0  \\
&&&&&& \\
} _.
$
}
\end{figure}

\begin{proof}

Throughout this proof we use the ordering $\ll$ on the atoms. Lemma \ref{L;product} and the multiplicative conditions in Definition \ref{colorsystem} imply that the products $a_{\gl_1} a_{\gl_2}$, and $a_{\gl_2} a_{\gl_3}$ are non-zero.  Now by (\ref{E;124}), (\ref{E;235}), and Lemma \ref{L;d} we have $d_A(- a_{\{\gl_1, \gl_2 , \gl_4\}}) = a_{\{\gl_1, \gl_2\}}$ and $d_A( - a_{\{\gl_2 , \gl_3 , \gl_5\}})= a_{\{\gl_2 , \gl_3\}}$.  These two facts now show that $d_1\left( a_{\gl_1} | a_{\gl_2} | a_{\gl_3} \right) = 0$, thus $a_{\gl_1} | a_{\gl_2} | a_{\gl_3}$ survives to the second page of the spectral sequence.  Using the multiplicative conditions in Definition \ref{colorsystem} and Lemma \ref{L;product}, $a_{\{\gl_1, \gl_2 , \gl_4\}}a_{\gl_3} = a_{\{ \gl_1 , \gl_2 , \gl_3 , \gl_4 \}}$.  Similarly, $a_{\{\gl_2 , \gl_3 , \gl_5\}}a_{\gl_1}=a_{\{ \gl_1 , \gl_2 , \gl_3 , \gl_5 \}}$, and hence $d_2 \left(a_{\gl_1} | a_{\gl_2} | a_{\gl_3} \right) = \left[ a_{\{ \gl_1 , \gl_2 , \gl_3 , \gl_4 \}} + a_{\{ \gl_1 , \gl_2 , \gl_3 , \gl_5 \}} \right] $. \end{proof}

If the class $\left[ a_{ \{\gl_1, \gl_2 , \gl_3, \gl_4\}} + a_{\{\gl_1, \gl_2 , \gl_3, \gl_5\}} \right]$ is non-zero on the $E_2$ page then it guarantees that there is a non-trivial Massey product, though this is not a necessary condition.  If this class is decomposable in cohomology it will be zero on the $E_2$ page but it can still represent a non-trivial Massey product, see Example \ref{E;MPs}. Actually, up to equivalence in the quotient by the ideal generated by $\left[a_{\gl_1}\right]$ and $\left[a_{\gl_3}\right]$ the cohomology class $\left[ a_{ \{\gl_1, \gl_2 , \gl_3, \gl_4\}} + a_{\{\gl_1, \gl_2 , \gl_3, \gl_5\}} \right]$ is equal to the Massey triple product $\left[ \left< a_{\gl_1},a_{\gl_2},a_{\gl_3}\right> \right] $.

The following corollaries provide conditions to determine if there are non-trivial Massey products.  The first corollary is a consequence of Theorem \ref{T;MP3}, Proposition \ref{P;MPs}.
\begin{corollary}

Let $\A$ be an edge colored hypergraphic arrangement satisfying the conditions of Theorem \ref{T;MP3}.  If the cohomology class $\left[ a_{ \{\gl_1, \gl_2 , \gl_3, \gl_4\}} + a_{\{\gl_1, \gl_2 , \gl_3, \gl_5\}} \right]$ is non-zero then $M(\A)$ admits a non-trivial Massey product.

\end{corollary}

The conditions of Theorem \ref{T;MP3} are easy to check using the edge colored hypergraph but checking that the desired cohomology class is non-zero is often quite difficult.  We supply sufficient conditions which guarantee that it's non-zero. The next corollary follows from Theorem \ref{T;MP3} and Lemma \ref{L;d}. 

\begin{corollary}\label{nocolors}

Let $\A$ be an edge colored hypergraphic arrangement satisfying the conditions of Theorem \ref{T;MP3}.  In addition, let $\Gamma :=\Lambda \bs \{\gl_1,\ldots ,\gl_5\}$.  If the set 
$$\{ \Psi \subseteq \Gamma \; | \; \Psi \Subset  \{\gl_1 , \gl_2, \gl_3 , \gl_4 \}  \text{ or } \Psi \Subset \{ \gl_1 , \gl_2 , \gl_3 , \gl_5 \}  \}$$
is empty then $M(\A)$ admits a non-trivial Massey product.

\end{corollary}

Next we illustrate a few edge colored hypergraphs that satisfy the conditions of Corollary \ref{nocolors} which can easily be generalized. One is an edge colored hypergraph where each color has only one edge. That is to say that the associated subspace arrangement is a hypergraph arrangement in the sense of Kozlov \cite{Koz-hyper}. The other is an edge colored hypergraph that admits a Massey color system where the underlying hypergraph is a graph, that is each edge contains exactly two vertices.

\begin{example} \label{E;MPs}

Let $(\Ho_1,\C_1)$ and $(\Ho_2,\C_2)$ be the edge colored hypergraphs of Figure \ref{ex1} and Figure \ref{ex2} respectively, where the edge color sets are given by $\gl_1=$ green, $\gl_2=$ red, $\gl_3=$ yellow, $\gl_4=$ blue, and $\gl_5=$ magenta. \begin{figure}[htbp]

\centerline {
\includegraphics[width=3in]{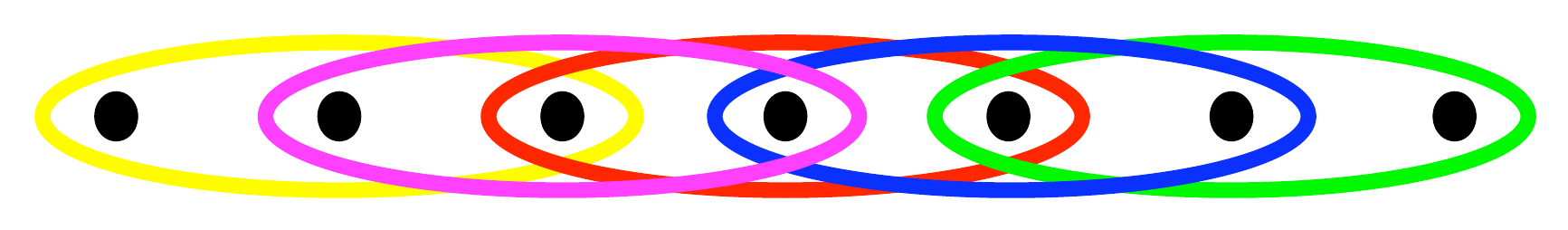}
}

\caption{$(\Ho_1,\C_1)$}
\label{ex1}
\end{figure} 

\begin{figure}[htbp]

\centerline {
\includegraphics[width=3in]{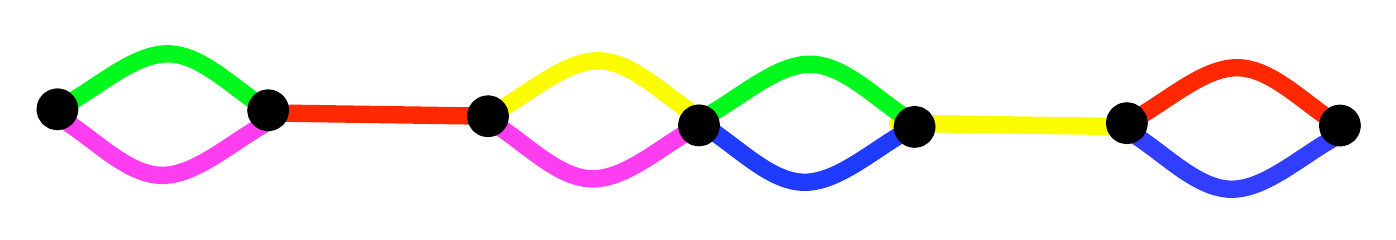}
}

\caption{$(\Ho_2,\C_2)$}
\label{ex2}
\end{figure}

Then for both of these edge colored hypergraphs, $\gl_1,\ \gl_2,\ \gl_3$ form a Massey color system with embedded colors $\gl_4$ and $\gl_5$. Further, both of these edge colored hypergraphs satisfy the hypothesis of Corollary \ref{nocolors}.  Notice also that, in both cases, the differential $d_2(a_{\gl_1} | a_{\gl_2} | a_{\gl_3} )$ is zero since
$$ [ a_{\lambda_{1234}} + a_{\lambda_{1235}} ] = [ a_{\lambda_{1345}} ]  = [ a_{\lambda_{14}} ] [ a_{\lambda_{35}} ] $$
and $d_1( a_{\lambda_{14}} | a_{\lambda_{35}} ) = [ a_{\lambda_{14}} ] [ a_{\lambda_{35}} ]$.

By Proposition \ref{P;pi} and a straight forward spectral sequence calculation we find that the rational homotopy groups for the examples above have the same ranks.  The following table gives the ranks in degrees less than eight. 
$$\begin{array}{c|c}
* & \text{rank} ( \pi^*)   \\ \hline
3 & 5 \\
4 & 4 \\ 
5 & 12  \\ 
6 & 16 \\ 
7 & 103 \\ 
\end{array}$$

\end{example}


\subsection{$k$-equal arrangements}

Let $(\Ho,\C)$ be the edge colored hypergraph where the set of edges is all possible subsets of size $k$ of the vertex set $[\ell]$ and each edge has its own color. Then the subspace arrangement corresponding to $(\Ho,\C)$ is known as the `$k$-equal arrangement' $\A_{\ell,k}$.
In this section we show that for some $k$ and $\ell$, all Massey products in the complex complement of the associated $k$-equal arrangements vanish.  Recall from Bj\"orner and Welker \cite{BW95} that the cohomology of the $k$-equal arrangement $\A_{\ell,k}$ is zero above degree $ \ell-1 + \lfloor \ell /k \rfloor (k-2) $ and the cohomology generators corresponding to the atoms are in degree $2k-3$.

\begin{theorem} \label{T;k-eq}

The arrangement $\A_{\ell,k}$ admits no non-trivial Massey products if
\begin{gather} \label{E;zero}
6k-9 > \ell + \lfloor \ell /k \rfloor (k-2).
\end{gather}

\end{theorem}

\begin{proof}
This theorem follows from degree counting.  Recall the differential in the Sinha-Walter spectral sequence
$$d_r : E_r^{-r,(r+1)(2k-3)} \to E_r^{0 , (r+1)(2k-3) - (r-1)}.$$
The degree of the target of this differential is greater than $\ell-1 + \lfloor \ell /k \rfloor (k-2)$ when 
$$ 2r(k-2) + 2k -1 > \ell + \lfloor \ell /k \rfloor (k-2). $$ 
Suppose that $r\geq 2$. Then by inequality (\ref{E;zero}) we have
$$2r(k-2) + 2k -1\geq 6k-9>\ell + \lfloor \ell /k \rfloor (k-2).$$
Thus for all $r \geq 2$ the target of the differential is zero. With Proposition \ref{P;MPs} this implies that any $r^{\textrm{th}}$ order Massey product has degree greater than $\ell-1 + \lfloor \ell /k \rfloor (k-2)$ and hence is zero.
\end{proof}

This theorem does not imply that such $k$-equal arrangements are formal but it does give evidence to support this claim.  The authors plan to use the ideas developed in this paper to further explore $k$-equal arrangements, Massey products, and the rational homotopy theory of edge-colored hypergraphic arrangements.


\bibliographystyle{amsplain}

\bibliography{bibsubspaces}

\end{document}